\documentclass[reqno,11pt]{amsart}

\usepackage{amsmath,amsthm,amssymb,amsfonts,color}
\usepackage{graphicx}
\usepackage{verbatim}
\usepackage[colorlinks=true,urlcolor=blue,
citecolor=red,linkcolor=blue,linktocpage,pdfpagelabels,bookmarksnumbered,bookmarksopen]{hyperref}
\usepackage[hyperpageref]{backref}

\newtheorem{theorem}{Theorem}[section]
\newtheorem{corollary}[theorem]{Corollary}
\newtheorem{lemma}[theorem]{Lemma}
\newtheorem{proposition}[theorem]{Proposition}
\theoremstyle{definition}
\newtheorem{remark}[theorem]{Remark}
\newtheorem{definition}[theorem]{Definition}

\newcommand{\CP}{\mathbb{C}P}
\newcommand{\HP}{\mathbb{H}P}
\newcommand{\CaP}{\mathrm{Ca}P^2}
\newcommand{\KP}{\mathbb{K}P}
\newcommand{\RH}{\mathbb{R}H}
\newcommand{\CH}{\mathbb{C}H}
\newcommand{\HH}{\mathbb{H}H}
\newcommand{\CaH}{\mathrm{Ca}H^2}
\newcommand{\KH}{\mathbb{K}H}

\def\S{\mathbb S}
\def\R{\mathbb R}
\def\C{\mathbb C}
\def\H{\mathbb H}
\def\Ca{\mathrm{Ca}}
\def\End{\mathrm{End}}
\def\K{\mathbb K}
\def\N{\mathbb N}

\def\div{{\rm div}} 
\newcommand{\si}{\mathrm{s}}
\newcommand{\co}{\mathrm{c}}
\newcommand{\tr}{\mathrm{Tr}}

\newcommand{\diam}{\mathrm{diam}}
\newcommand{\dd}[2]{\frac{\partial #1}{\partial #2}}
\numberwithin{equation}{section}

\begin{document}

\title[A spectral characterization of geodesic balls in non-compact ROSS]{A spectral characterization of geodesic balls in non-compact rank one symmetric spaces}

\author{Philippe Castillon, Berardo Ruffini}

\subjclass[2010]{46E35, 35P30, 39B72}

\keywords{Brock-Weinstock inequality, harmonic manifold, isoperimetric inequality}

\begin{abstract}
	In constant curvatures spaces, there are a lot of characterizations of geodesic balls as optimal domain for shape optimization problems. Although it is natural to expect similar characterizations in rank one symmetric spaces, very few is known in this setting.
	
	In this paper we prove that, in a non-compact rank one symmetric space, the geodesic balls uniquely maximize the first nonzero Steklov eigenvalue among the domains of fixed volume, extending to this context a result of Brock in the Euclidean space. Then we show that a stability version of the ensuing Brock-Weinstock inequality holds. The idea behind the proof is to exploit a suitable weighted isoperimetric inequality which we prove to hold true, as well as in a stability form, on harmonic manifolds.
	
	Eventually we show that, in general, the geodesic balls are not global maximizers on the standard sphere.
\end{abstract}

\maketitle

%
%
%
%
\section{Introduction}

%
\subsection*{Shape optimization problems}
A shape optimization problem on a Riemannian manifold is simply an optimization problem of the form
\[
\min_{\Omega\in\mathcal A} F(\Omega)\ \ \ \ \mbox{ or }\ \ \ \ \max_{\Omega\in\mathcal A} F(\Omega)
\]
where the class of optimization $\mathcal A$ is a subset of the  powerset of the ambient space and $F$ is a functional on $\mathcal A$. The most famous instance of this kind of problems is the isoperimetric problem where $F(\Omega)$ is the volume of $\partial\Omega$ and $\mathcal{A}$ is a class of  domains of fixed measure. Another class is constituted by the  spectral optimization problems	where $F$  depends on the spectrum of an elliptic operator. In this class the archetype example  is the minimization of $\lambda_1(\Omega)$, the first eigenvalue of the Dirichlet-Laplacian operator, under a volume constraint. Its solution (due to Faber and Krahn) tells us that, on $\R^n$, $\lambda_1(\Omega)$ is minimized by Euclidean balls. Another famous result is the Sz\"ego-Weinberger inequality, stating that among sets of prescribed volume, the ball maximize the first non-zero eigenvalue of the Neumann-Laplacian. See \cite{Henrot} for a comprehensive guide on spectral optimization problems in flat spaces.

These optimization problems were first considered in the Euclidean space and then extended to constant curvature spaces, giving rise to variational characterization of geodesic balls in the real hyperbolic space $\R H^n$ and in the round sphere $\S^n$ (see for example \cite[chapter 2]{Burago-Zalgaller} and \cite[chapter 6]{Chavel} for the isoperimetric and Faber-Krahn inequalities in constant curvature spaces).

Other Riemannian manifolds where this kind of problems are natural are the rank one symmetric spaces (noted ROSS in the sequel): these spaces are two-point homogeneous, that is, for each two couple of points $(x,y)$ and $(x',y')$ such that $d(x,y)=d(x',y')$ there exists an isometry that  brings $x$ to $x'$ and $y$ to $y'$. This implies that their geodesic spheres are homogeneous, and in turn suggests that their geodesic balls are good candidate for being optimal domains of isoperimetric and spectral problems (in particular, the geodesic spheres have constant mean curvature). However, beside the constant curvature case, almost nothing is known for the isoperimetric problem or for the spectral optimization problems cited above (cf. \cite[section 7.1.2]{Berger} for a discussion on the isoperimetric problem in compact ROSS).

%
\subsection*{The Steklov optimization problem}
Let $M$ be a Riemannian manifold of dimension $m$ and let $\Omega\subset M$ be an connected open bounded set with Lipschitz boundary. The Steklov eigenvalue problem consists in finding the real numbers $\sigma$ for which the boundary value problem
\[
\begin{cases}
-\Delta u=0&\text{in $\Omega$}\\
\partial_\nu u=\sigma u&\text{on $\partial\Omega$}
\end{cases}
\] 
has a non-trivial solution $u$, where $\Delta$ is the Laplace-Beltrami operator on $\Omega$ and $\partial_\nu u$ is the normal derivative of $u$ on  $\partial\Omega$. The Steklov eigenvalues of $\Omega$ form an increasing sequence $0=\sigma_0(\Omega)<\sigma_1(\Omega)\le\sigma_2(\Omega)\le\dots$ diverging  to $+\infty$ (see section \ref{section-Steklov} for a more comprehensive introduction to the spectrum of the Steklov operator).

Regarding shape optimization problems for the Steklov eigenvalues of Euclidean domains, a first result was given  in the dimension $2$ by Weinstock in $1954$ \cite{Weinstock}, who showed that the ball maximizes $\sigma_1(\Omega)$ among simply connected domains with prescribed perimeter (thanks to the isoperimetric inequality, this turns out to imply the same result among domains with fixed volume). In $2001$, Brock showed the same result without topological nor dimensional constraints, among domains with fixed volume \cite{Brock}. Later on, in $2012$, Brasco, De Philippis and the second author exhibited a new simple proof of the Brock-Weinstock inequality, which allowed them to get a stability version of it \cite{Brasco-DePhilippis-Ruffini}.

It is worth recalling that the problem of maximizing the first Steklov eigenvalue with perimeter constraint has been recently solved in any dimension among convex sets \cite{bucuretal} but it is open in dimension greater than $2$ in its full generality. In dimension $2$, it is known that the ball is not a global minimizer. See \cite{Henrot} for the state of the art on spectral shape optimization problems in the Euclidean space.

In a non-Euclidean setting, the Steklov problem was mainly studied as an optimization problem on the space of Riemannian metrics: given a differentiable manifold $\Omega$ with boundary, find the metric maximizing $\sigma_1(\Omega)$ under a volume or perimeter constraint. See \cite{Girouard-Polterovich} for a recent survey on the Steklov eigenvalue problem in Riemannian geometry.
	Weinstock-like inequalities in Riemannian manifolds were first considered by J.F. Escobar in \cite{Escobar1,Escobar2,Escobar3} and later by Binoy and G. Santhanam in \cite{Binoy-Santhanam} where the authors prove that the geodesic balls maximize $\sigma_1(\Omega)$ among domains of fixed volume in non-compact ROSS.

%
\subsection*{The main results}

	The aim of this paper is to prove the quantitative version of the Brock-Weinstock inequality in non-compact ROSS. Moreover, we investigate the compact case and prove that this inequality does not hold on the sphere by computing the first Steklov eigenvalue of a spherical strip.
	The main theorem is the following quantitative inequality:
	\begin{theorem}\label{th:main1}
		Let $M$ be a non-compact ROSS. For any $v>0$, there exists a positive constant $C=C(M,v)$ such that, for any domain $\Omega\subset M$ with $|\Omega|=v$ we have
		\[
		\sigma_1(\Omega)\left(1+C|\Omega\setminus B|^2\right)\le\sigma_1(B),
		\]
		where $B$ is a geodesic ball with $|B|=|\Omega|$.
	\end{theorem}

As a consequence, we get the Brock-Weinstock inequality in non-compact ROSS, provided of equality cases.

The main point in order to prove Theorem \ref{th:main1} is that, following the proof in \cite{Brock}, the maximality of the ball for the first Steklov eigenvalue ensues from a {\em  weighted isoperimetric inequality}, which takes the form, in the euclidean setting,
\[
\int_{\partial\Omega}|x|^2\,dx\ge \int_{\partial B}|x|^2\,dx,
\]
whenever $B$ is the ball centered at the origin of the same measure as $\Omega$. 

	This kind of isoperimetric inequalities were first considered in \cite{B-B-M-P}, and then used in \cite{Brasco-DePhilippis-Ruffini} to prove the quantitative Brock-Weinstock inequality in the Euclidean space.
The proof after this idea, after  suitable modifications, works as well in the setting of ROSS. Thus the problem reduces to  prove  (a suitable formulation of) the weighted isoperimetric inequality in a Riemannian setting. This is precisely the core of the following result which holds on harmonic manifolds (see section 2.2 for their definition and main properties) and thus on ROSS.
\begin{theorem}\label{th:main2}
	Let $M$ be an harmonic manifold and let $o\in M$ be some fixed point. For any domain $\Omega\subset M$, if $B$ is the ball centered in $o$ with radius $R$ such that $|B|=|\Omega|$, then we have
	\[
	P_o(\Omega)=\int_{\partial\Omega}\left(\frac{v(r)}{v'(r)}\right)^2\,dv_{\partial\Omega} \ge 	\int_{\partial B}\left(\frac{v(R)}{v'(R)}\right)^2\,dv_{\partial B}=P_o(B),
	\]
	where $r=d(o,.)$ is the distance function to $o$, and $v(t)=|B_t|$ is the volume of a ball of radius $t$. Moreover, equality holds if and only if $\Omega=B$.
\end{theorem}
Notice that on harmonic manifolds (and thus on ROSS) the volume of a ball does not depend on the center (cf. section 2.2). 

Once the isoperimetric inequality, and as a consequence, the Brock-Weinstock inequality, are settled down, a natural question is that of the uniqueness of the solution. In this paper we are able to show the following quantitative version of the weighted isoperimetric inequality which, as a consequence, gives rise to the quantitative Brock-Weinstock inequalities of Theorem \ref{th:main1}. 
\begin{theorem}\label{th:main3}
	Let  $M$ be a non-compact harmonic manifold and $o\in M$ be some fixed point. Let $\Omega\subset M$ be a domain and $B$ be the ball centered in $o$ such that $|B|=|\Omega|$. Then there exists a constant $C=C(|\Omega|,M)$ such that 
	$$
	P_o (\Omega)-P_o (B)\ge C|\Omega\setminus B|^2.
	$$
\end{theorem}
Notice that the constant $C$ depends on $\Omega$ only via its volume, so that the inequality holds true with a fixed constant on the class of domains of given volume.

Eventually, we investigate the compact setting. In this case we  show that the Brock-Weinstock inequality does not hold, at least in its full generality. Precisely we prove the following result:
\begin{theorem}\label{th:main4}
	Let $M=\mathbb S^m$. There exists  $R>\frac{\pi}{2}$ such that $\sigma_1(\Omega_R)>\sigma_1(B_S)$, where $\Omega_R$ is the intersection of two geodesic balls of radius $R$ with antipodal centers, and $B_S$ is a geodesic ball such that $|\Omega_R|=|B_S|$.
\end{theorem}	

%
\subsection*{Plan of the paper}
In Section \ref{sec:background} we settle down the notations used throughout the paper and we properly introduce the Steklov spectrum on a manifold, as well as a variational characterization of its first eigenvalue. Thereafter we recall the main geometric features of harmonic manifolds and ROSS.

Section \ref{sec:spectruminross} is devoted to the study the Steklov spectrum of geodesic balls in ROSS. In particular, for non-compact ROSS, we determine the  eigenfunctions corresponding to the first eigenvalue of the ball which, in the Euclidean setting, are the coordinate functions. The precise knowledge of these eigenfunctions is a crucial point for our proofs to work.

Section \ref{sec:steklov} contains the proof of Theorems \ref{th:main1}, \ref{th:main2}, \ref{th:main3}, while the proof of Theorem \ref{th:main4} is postponed to Section \ref{sec:nottheball}.
%
%
%
%
%
%
\section{Background and notations}\label{sec:background}
	
%
\subsection{The Steklov spectrum of a domain}\label{section-Steklov}
Let $M$ be a Riemannian manifold of dimension $n$ and let $\Omega\subset M$ be an connected open bounded set with Lipschitz boundary. The Steklov eigenvalue problem,  introduced by the Russian mathematician V. A. Steklov \cite{Kuznetsov-others}, consists in finding a solution $u$ of the boundary value problem
\[
\begin{cases}
-\Delta u=0&\text{in $\Omega$}\\
\partial_\nu u=\sigma u&\text{on $\partial\Omega$},
\end{cases}
\] 
where $\Delta$ is the Laplace-Beltrami operator on $\Omega$ and $\partial_\nu u$ is the partial derivative of $u$ on  $\partial\Omega$.

Equivalently, the problem reduces to study the spectrum of the operator Dirichlet-to-Neumann $\mathcal R:L^2(\partial \Omega)\to L^2(\partial \Omega)$ which maps $f\in L^2(\partial\Omega)$ to the normal derivative on the boundary of the harmonic extension of $f$ inside of $\Omega$. 
This operator is symmetric and positive. Moreover, thanks to the compactness of the embedding of the trace operator from $H^1(\Omega)$ into $L^2(\partial\Omega)$, the resolvent of $\mathcal R$ is compact. Thus $L^2(\partial\Omega)$ admits an Hilbert basis $\{u_k\}_{k\in\mathbb N}$ of eigenfunctions for $\mathcal R$, and of positive eigenvectors $0=\sigma_0(\Omega)<\sigma_1(\Omega)\le \sigma_2(\Omega)\le\dots$ diverging to $+\infty$ such that
\[
\begin{cases}
-\Delta u_k=0&\text{in $\Omega$}\\
\partial_\nu u_k=\sigma_k(\Omega)u_k&\text{on $\partial\Omega$}.
\end{cases}
\]
For ease of presentation, for  $u\in H^1(\Omega)$ we still denote  by $u$ its trace in $L^2(\partial\Omega)$. Here solutions are intended in the weak$-H^1(\Omega)$ sense, that is
\[
\int_{\Omega}\langle\nabla u_k,\nabla\varphi\rangle\,dv_M=\sigma_k(\Omega)\int_{\partial\Omega}u_k\varphi\,dv_{\partial\Omega}\qquad\text{for every $\varphi\in H^1(\Omega)$}.
\]
The eigenfunctions $u_k$ and the eigenvalues $\sigma_k(\Omega)$ are respectively the  Steklov eigenfunctions and eigenvalues of $\Omega$.
In this paper we are mostly interested in the first non-zero Steklov eigenvalue, $\sigma_1(\Omega)$ which  can be characterized variationally as 
\begin{equation}\label{steklovvariational}
\sigma_1(\Omega)=\min\left\{\frac{\displaystyle\int_{\Omega}|\nabla u|^2\,dv_M }{\displaystyle\int_{\partial\Omega}u^2\,dv_{\partial\Omega}}:u\in H^1(\Omega),\,\int_{\partial\Omega}u\,dv_{\partial\Omega}=0\right\}.
\end{equation}

%
\subsection{Harmonic manifolds and rank one symmetric spaces}\label{subsecharm}
Although our main result only concern the ROSS, some of the lemmas or intermediate results hold true for the larger class of \emph{harmonic manifolds}.

Harmonic manifolds are those Riemannian manifolds whose harmonic functions have the mean value property. Equivalently, a Riemannian manifold $M$ is harmonic if and only if there exists a function $h:\R_+^*\to\R$ such that any sphere of radius $r$ has constant mean curvature $h(r)$. Another equivalent property is that there exists a function $\theta:\R_+\to\R$ such that at any point $x\in M$ the volume form in normal coordinates reads $dv_M=\theta(r)drd\xi$ where $d\xi$ is the canonical volume form of the unit tangent sphere $U_xM$ at $x$ (the function $\theta$ is usually called the volume density function of $M$). It is not difficult to show that harmonic spaces are Einstein manifolds and that a Riemannian manifold is harmonic if and only if its universal cover is harmonic. Therefore, in what follows we will only consider simply connected harmonic manifolds. General properties of harmonic manifolds can be found in \cite[chapter 6]{Besse}, \cite[sections 1 and 2]{Szabo} and \cite[section 2]{Knieper-Peyerimhoff}.

The basic examples of harmonic manifolds are the Euclidean spaces and the Rank One Symmetric Spaces (ROSS, see Section \ref{sec:ROSS} for the definition and main properties). In 1944, A. Lichn\'erowicz conjectured, and proved in dimension 4, that the Euclidean spaces and ROSS are the only harmonic manifolds. The conjecture was then proved  by Z.I. Szabo for compact simply connected manifolds (cf. \cite{Szabo}). However, the Lichn\'erowicz conjecture was proved to be false in the non-compact case: E. Damek and F. Ricci constructed harmonic homogeneous manifolds which are not ROSS (cf. \cite{Damek-Ricci}). Up to now, the Euclidean spaces, ROSS and Damek-Ricci spaces are the only known harmonic manifolds.

The main properties of harmonic manifolds we will use in our proofs are summarized in the following proposition:
\begin{proposition}\label{prop-harm_mfd_properties}
	Let $M^m$ be a non-compact harmonic manifold and note $\theta(r)$ its density function and $h(r)$ the mean curvature of spheres of radius $r$. The following holds:
	\begin{enumerate}
		\item[$(i)$] $(m-1)h(r)=\frac{\theta'(r)}{\theta(r)}$.
		\item[$(ii)$] $h(r)$ decreases to a nonnegative constant $h_0$.
		\item[$(iii)$] $(m-1)h_0$ is the volume entropy of $M$.
		\item[$(iv)$] $h_0=0$ if and only if $M$ is the euclidean space.
	\end{enumerate}
\end{proposition}
\begin{proof}
	The first point is a classical fact of Riemannian geometry.
	
	The second point is proved in \cite{Ranjan-Shah-2}. In particular, the non-compact harmonic manifolds have no conjugate points, so that they have well defined horospheres and Busemann functions, and their horospheres have constant mean curvature $h_0$.
	
	The volume entropy of $M$ is $\lim_{r\to+\infty}\frac{\ln(v(r))}{r}$, where $v(r)$ is the volume of a ball of radius $r$. The manifold $M$ being harmonic, the classical derivation formula gives $v''(r)=(m-1)h(r)v'(r)$ and the third point follows from $\lim_{r\to+\infty}h(r)=h_0$.
	
	The last point is a consequence of results by Y. Nikolayevsky \cite[Theorem 2]{Nikolayevsky} and A.~Ranjan and H. Shah \cite[Theorem 4.2]{Ranjan-Shah}
\end{proof}

%
\subsection{The geometry of ROSS}\label{sec:ROSS}
In this section we give some geometric properties of ROSS which will be used later. In particular, we describe the extrinsic and intrinsic geometry of their geodesic spheres. There are two families of simply connected ROSS, the compact ones (the round sphere $\S^n$ and the projective spaces $\CP^n$, $\HP^n$ and $\CaP$) and the non-compact ones (the hyperbolic spaces $\RH^n$, $\CH^n$, $\HH^n$ and $\CaH$). We outline here the definition of ROSS and refer to \cite[Chapter 3]{Besse} for the details of the construction.

In what follows, $\K$ will denote one of the following: the field $\R$ of real numbers, the field $\C$  of complex numbers, the algebra $\H$ of quaternions or or the algebra $\Ca$ of  octonions. Let $d=\dim_\R(\K)$, and consider $\K^{n+1}$ equipped with the Hermitian product
$$
\langle x,y \rangle = \sum_{i=0}^n x_i\bar{y_i}
$$
whose real part $\langle.,.\rangle_\R$ is the standard Euclidean inner product on $\K^{n+1}$ viewed as an $dn+d$ real vector space. If $\K\not=\Ca$, the group $U(1,\K)$ of unit elements in $\K$ acts on the unit sphere $\S^{nd+d-1}$ of $\K^{n+1}$ by right multiplication, and the projective space $\KP^n$ is $\S^{nd+d-1}/U(1,\K)$. Therefore it is the base space of the fibration
$$
\S^{d-1} \to \S^{nd+d-1} \to \KP^n
$$
and, considering the standard metric on $\S^{nd+d-1}$, there is a unique metric on $\KP^n$ which makes this fibration a Riemannian submersion with totally geodesic fibers (cf. \cite[section 2.A.5]{Gallot-Hulin-Lafontaine}). In case $\K=\Ca$, such a construction only works if $n=2$ (cf. \cite[section 3.G]{Besse})

The non-compact ROSS are defined in a similar way, replacing the hermitian product on $\K^{n+1}$ by
$$
\langle x,y \rangle = -x_0\bar{y_0} + \sum_{i=1}^n x_i\bar{y_i}
$$
whose real part $\langle.,.\rangle_\R$ is a real bilinear form with signature $(d,nd)$ on $\K^{n+1}$, and replacing the sphere $\S^{nd+d-1}$ by $H^{nd+d-1}=\{x\in\K^{n+1}\ |\ \langle x,x \rangle = -1 \}$. We still have the $U(1,\K)$ action on $H^{nd+d-1}$ and the hyperbolic space $\KH^n$ is $H^{nd+d-1}/U(1,\K)$. It is the base space of the fibration
$$
\S^{d-1} \to H^{nd+d-1} \to \KH^n.
$$
The metric induced by  $\langle.,.\rangle_\R$ on $H^{nd+d-1}$ has signature $(d-1,nd)$, and since it is preserved by the action of $U(1,\K)$ whose orbits are $d-1$ dimensional spheres, its restriction to the orthogonal of the fiber is positive definite. Therefore the fibration induces a Riemannian metric on $\KH^n$ (cf. \cite[Chapter XI, example 10.7]{Kobayashi-Nomizu-II} for such a construction of the complex hyperbolic space).

Let $M^m$ be a ROSS of dimension $m=dn$, with $d=\dim_\R(\K)$. 
For a common treatment of the compact and non-compact cases we will use the following notations :
$$
\si(t)=\left\{\begin{array}{ll}
						\sin(t) & \mbox{if } M \mbox{ is a compact ROSS} \\
						\sinh(t) & \mbox{if } M \mbox{ is a non-compact ROSS} \\
					\end{array}\right.
$$
and
$$
\co(t)=\left\{\begin{array}{ll}
						\cos(t) & \mbox{if } M \mbox{ is a compact ROSS} \\
						\cosh(t) & \mbox{if } M \mbox{ is a non-compact ROSS} \\
					\end{array}\right.
$$

We will use the framework of Jacobi tensor to describe the second fundamental form of spheres and the density function. Let $\gamma$ be a geodesic with initial point $x=\gamma(0)$ and initial speed $\xi=\dot{\gamma}(0)\in U_xM$, and note $N\gamma$ its normal bundle, that is the disjoint union of the $N_t\gamma=\{\eta\in T_{\gamma(t)}M\ |\ \langle \eta,\dot{\gamma}(t)\rangle=0\}$. A $(1,1)-$tensor $A$ along $\gamma$ is a differentiable section of $\End(N\gamma)$, whose derivative is defined by $A'X = D_{\dot{\gamma}}(AX) - AD_{\dot{\gamma}}X$, for any normal vector field $X$ along $\gamma$.

In particular, the curvature tensor $R$ of $M$ induces a $(1,1)-$tensor $R(t)$ along $\gamma$ defined by $R(t)\eta = R(\dot{\gamma}(t),\eta)\dot{\gamma}(t)$, and a Jacobi tensor is a solution of the equation $A''(t)+R(t)A(t)=0$. In the sequel, we will note $A_\xi$ the Jacobi tensor with initial conditions $A_\xi(0)=0$ and $A'_\xi(0)=I$.

If $X$ is a parallel vector field along gamma with $X(0)=\eta$ orthogonal to $\xi$, then $Y=A_\xi X$ is the Jacobi field along $\gamma$ with $Y(0)=0$ and $Y'(0)=\eta$. As a consequence we have that the density function is given by $\theta(\xi,t)=\det(A_\xi(t))$, and the second fundamental form of the sphere $S_t(x)$ at $\exp_x(t\xi)$ is given
by $A'_\xi(t)A^{-1}_\xi(t)$ (cf. for example \cite[section 3.H]{Gallot-Hulin-Lafontaine} for the computation of the density from Jacobi fields).

As a consequence of the definition of $M$, it carries $d-1$ orthogonal complex structures $J_1,\dots,J_{d-1}$ (cf. \cite[Chapter 3]{Besse}). The curvature tensor of $M$ is described, using these complex structures, in the following way: for any $\xi,\eta\in U_xM$ with $\eta$ orthogonal to $\xi,J_1\xi,\dots,J_{d-1}\xi$ we have
$$
R(\xi,J_i\xi)\xi = 4\varepsilon J_i\xi\ \ \ \mbox{ and }\ \ \ 
R(\xi,\eta)\xi = \varepsilon \eta.
$$
where we choose $\varepsilon\in\{-1,1\}$ being $1$ in the compact case and $-1$ in the non-compact case. For a geodesic $\gamma$ with initial point
$x=\gamma(0)$ and initial speed $\xi=\dot{\gamma}(0)\in U_xM$,
consider an orthonormal parallel frame $(E_1,\dots,E_m)$ such that
$E_1=\dot{\gamma}$ and, for $i=1,\dots,d-1$, $E_{i+1}=J_i\dot{\gamma}$.
For each $t$, the eigenspaces of $R(t)$ are spanned by $E_2(t),\dots,E_d(t)$
with eigenvalue $4\varepsilon$ and by $E_{d+1}(t),\dots,E_m(t)$ with eigenvalue $\varepsilon$.
Therefore, integrating the Jacobi equation, we have that $A_\xi(t)$ has the same
eigenspaces with eigenvalues
$$
\alpha(t)=\frac{1}{2}\si(2t)=s(t)c(t)\ \ \ \mbox{ and }\ \ \ 
\beta(t)=\si(t),
$$
with respective multiplicities $d-1$ and $m-d$. {From this computation of Jacobi tensors, we have that the non-compact ROSS have no conjugate point, while the compact ones have a first conjugate point at distance $t=\pi$ if $M=\S^m$ and $t=\frac{\pi}{2}$ otherwise. In particular, for compact ROSS, the injectivity domain of the exponential map is the ball of radius $\diam(M)$ in $U_xM$. In the following, we assume that $t<\diam(M)$.}

As a first consequence of the above computation, the density function of $M$ is given by
$$
\theta(t)=\frac{1}{2^{d-1}}\si(2t)^{d-1}\si(t)^{m-d}=c(t)^{d-1}s(t)^{m-1}.
$$
Moreover, the second fundamental form of the geodesic sphere $S_x(t)$ has two eigenvalues,
$2\frac{\co(2t)}{\si(2t)}$ and $\frac{\co(t)}{\si(t)}$, the first one being of multiplicity $d-1$ with an eigenspace spanned by $J_1\dd{}{r},\dots,J_{d-1}\dd{}{r}$, where $\dd{}{r}$ is the radial field centered at $x$. From these computations we get that the mean curvature of the geodesic spheres of radius $r$ satisfies
\begin{equation}\label{eqn-mean_curvature}
	h(t)= \frac{c(t)}{s(t)} - \varepsilon\frac{d-1}{m-1}\frac{s(t)}{c(t)}.
\end{equation}

\begin{remark}
	If $M$ is the sphere $\S^m$ or the real hyperbolic space $\RH^m$ then only the second eigenvalue appears and the geodesic spheres of $M$ are totally umbilical.
	
	For a general ROSS, two important properties will be used later: the eigenvalues of the second fundamental form only depend on the radius of the geodesic sphere and there exists a parallel orthonormal frame of eigenvectors along the geodesic $\gamma$.
	
	These two properties are not satisfied in general Damek-Ricci spaces which are not two-point homogeneous. Moreover, their curvature tensor $R(t)$ do not have parallel eigenvectors nor constant eigenvalues along the geodesic $\gamma$ (cf. \cite[section 4.3]{Berndt-Ticerri-Vanhecke}).
\end{remark}

From the previous computations we can derive the induced metric of geodesic spheres. Let $o$ be some fixed point in $M$, and $g_c$ the canonical metric of the unit sphere $U_oM$ in $T_oM$. For each $r>0$, consider the exponential map
$$
\left\{\begin{array}{rcl}
U_oM & \to & S_r(o) \\
\xi & \mapsto & \exp_o(r\xi)
\end{array}\right.
$$
and denote by $g_r$ the pullback on the unit sphere $U_oM$ of the metric of the geodesic sphere $S_r(o)$. 

Using the complex structures $J_1,\dots,J_{d-1}$ on $T_oM$, consider the $d-1$ unit vector
fields on $U_oM$ defined at $\xi\in U_oM$ by $J_1\xi,\dots,J_{d-1}\xi$, and their dual 1-forms $\delta_k$. As the differential of the exponential map is given by Jacobi tensor, from the computations above we have that the metric $g_r$ is given by
$$
g_r=\si^2(r)\left(g_c-\sum_{k=1}^{d-1}\delta_k\otimes\delta_k+\co^2(r)\sum_{k=1}^{d-1}\delta_k\otimes\delta_k\right)
$$
\begin{remark}
	If $M$ is the sphere $\S^m$ or the real hyperbolic space $\RH^m$ then $g_r=s^2(r)g_c$ and the geodesic sphere $S_o(r)$ is a round sphere of curvature $\frac{1}{s^2(r)}$.
	
	For the other ROSS, the metrics $g_r$ are known as Berger's metric and were widely studied, in particular for their spectral properties.
\end{remark}

%
%
%
%
\section{The Steklov spectrum of geodesic balls in ROSS}\label{sec:spectruminross}

In this section we introduce the main definitions and some preliminary results which will be exploited in Section \ref{sec:steklov}.

Let $M$ be a ROSS and $o$ be some fixed point in $M$. For any $x\in M$, let
$r(x)=d(o,x)$ and  (assuming that $r(x)<\diam(M)$ if $M$ is a compact ROSS) let $w(x)\in U_oM$
be the unique unit vector at $o$ such that $x=\exp_o(r(x)w(x))$.

We will note $\Delta_r$ and $\Delta_c$ the Laplace operators of the metrics $g_r$ and $g_c$
respectively.

%
\subsection{The Laplace spectrum of geodesic spheres}

Up to the factor $\si^2(r)$, the metric $g_r$ is a Berger metric on the sphere for which the spectrum
of the Laplacian is known. The key point in describing the spectrum is that the sphere $U_oM$
is the total space of a Riemannian submersion and that the metric $g_r$ is obtained from the
canonical metric $g_c$ by a rescaling of the fibers. This setting was considered in the special
case of odd dimensional spheres in \cite{Tanno1,Tanno2} and in the general case of a
Riemannian submersion in \cite{Berard-Bergery-Bourguignon}.
\begin{theorem}
	The Hilbert space $L^2(U_oM)$ admits a Hilbert basis which consists of eigenfunctions of $\Delta_c$
	and of each $\Delta_r$.
\end{theorem}
\begin{proof}
	Multiplying the Berger metric by $\si^2(r)$ rescale the eigenvalues of the Laplacian by $\si^{-2}(r)$, but does not modify the eigenfunctions. Therefore the theorem is just a consequence of \cite{Berard-Bergery-Bourguignon} Theorem 3.6 and Corollary 5.5.
\end{proof}

In the sequel we consider a Hilbert basis $(f_k)_{k\in\N}$ of common eigenfunctions and we note
$\lambda_k(r)$ the eigenvalue of $\Delta_r$ associated to $f_k$, and $\lambda_{c,k}$ the
eigenvalue of $\Delta_c$ associated to $f_k$. Moreover, we chose the Hilbert basis in such a way that
the sequence $(\lambda_{c,k})_{k\in\N}$ is the increasing sequence of eigenvalues of the standard
unit sphere.

As the functions $(f_k)_{k\in\N}$ are eigenfunctions of $\Delta_c$, they are given by spherical harmonics
of the tangent space $T_oM$ (i.e. they are the restrictions to the unit sphere $U_oM$ of homogeneous
harmonic polynomials of some Cartesian coordinates in $T_oM$). In particular, we have that
$f_0$ is a constant function, and, for $k=1,\dots,m$, $f_k=\langle \xi_k,.\rangle$ where
$(\xi_1,\dots,\xi_m)$ is an orthonormal basis of $T_oM$.

\begin{remark}
	The ordering of the eigenvalues may change as $r$ varies. In particular, if $M$ is a compact
	ROSS, there exists a parameter $r_0$ such that the eigenvalues of $\Delta_{r_0}$ associated to
	the first and some of the second spherical harmonics coincide, and such that, for $r>r_0$,
	the smallest non-zero eigenvalue of $\Delta_r$ is given by some second spherical harmonics (cf.
	\cite{Tanno1,Tanno2,Berard-Bergery-Bourguignon}). This kind of phenomena was the main
	motivation for studying the spectrum of Berger spheres.
\end{remark}

\begin{remark}
	The fact that the first spherical harmonics give rise to eigenfunctions of $\Delta_r$ also
	holds true on harmonic manifolds (cf. \cite{Knieper-Peyerimhoff, Ranjan-Shah}).
	However, as explained
	in the above remark, it is not true in general that they give rise to the smallest non-zero
	eigenvalue of $\Delta_r$.
\end{remark}

%
\subsection{From the Laplace spectrum of spheres to the Steklov spectrum of balls}

In this section we construct a family of harmonic functions from the eigenfunctions of $\Delta_r$.
These functions can be seen as the ``spherical harmonics'' of $M$ and will give rise to
Steklov eigenfunctions of the geodesic balls.

\begin{proposition}\label{prop-Spherical_harm}
	For each $k\in\N$ there exists a function $a_k:[0,\diam(M))\to\R$ such that the function $F_k:M\to\R$ defined by $F_k(x)=a_k(r(x))f_k(w(x))$ is harmonic. Moreover, $a_k$ admits a power series expansion  of the form
	\[
	a_k(r) = r^p + \sum_{i\ge p+1} \xi_ir^i,
	\]
	where $p$ depends on $k$ and satisfies $p\ge 1$ if $k\ge 1$.
\end{proposition}
\begin{proof}
	Let $a:\R_+\to\R$ be a smooth function. Using that the gradient of $a(r)$ is orthogonal
	to the gradient of $f_k(w)$ we have
	$$
	\Delta (a(r)f_k(w)) = a(r)\Delta f_k(w) + f_k(w)\Delta a(r).
	$$
	Since $f_k(w)$ is constant along the geodesic lines from $o$ we get, at any point $x\in M$,
	$$
	\Delta f_k(w) = \tr(D^2f_k(w)_{|_{T_xS_r}}) = \Delta^{S_r}f_k(w) = -\lambda_k(r)f_k(w).
	$$
	On the other hand, because $\Delta r= (m-1)h(r)$ we have
	$$
	\Delta a(r) = (m-1)a'(r)h(r) + a''(r).
	$$
	Finally we get
	\begin{equation}\label{eqn-ODE_harmonic}
		\Delta\bigl(a(r)f_k(w)\bigr) = \bigl(a''(r)+(m-1)h(r)a'(r)-\lambda_k(r)a(r)\bigr)f_k(w).
	\end{equation}
	To get the result we only need to show that a solution $a_k$ of  the Sturm-Liouville ODE
	\begin{equation}\label{eqn-ODE_harmonic2}
	a''+(m-1)ha'-\lambda_ka=0
	\end{equation}
	exists. This is likely to be folklore, but since we could not find a precise reference, we sketch a way to get such a solution. First we rewrite our equation as
	\begin{equation}\label{liouville}
	r^2a''(r)+\left(r\frac{\theta'}{\theta}(r)\right)ra'(r)-r^2\lambda_k(r)a(r)=0,
	\end{equation}
	and we recall that both $r\frac{\theta'}{\theta}(r)$ and $r^2\lambda_k(r)$ admit power series expansions converging on $[0,\diam(M))$:
	\[
	r\frac{\theta'}{\theta}(r)=\sum_{i\ge0} \alpha_ir^i\ \ \ \mbox{ and }\ \ \  r^2\lambda_k(r)=\sum_{i\ge0} \beta_ir^i.
	\]
	Writing the Sturm-Liouville ODE in the form \eqref{liouville},  emphasizes the singular nature of the equation at $r=0$, see \cite[Chapter V]{CourantHilbert}.
	
	{In the neighborhood of any point, the metric of a Riemannian manifold is locally asymptotically Euclidean. Therefore,} the first terms of the expansions of $r\theta'/\theta$ and $r^2\lambda_k(r)$ are $\alpha_0=m-1$, $\alpha_1=0$, $\beta_0=\lambda_{c,k}$, and $\beta_1=0$, where $\lambda_{c,k}$ is the $k-$th eigenvalue of the standard sphere $\S^{m-1}$. In particular, there exists $p\in\N$ such that $\beta_0=p(p+m-2)$, with $p\ge 1$ if $k\ge 1$. We begin by searching a formal solution of Equation \eqref{liouville}, with the goal to show later that this is indeed a solution. In other words, we write a solution as  $a(r)=\sum_{i\ge0}\xi_ir^i$. Plugging it into the equation \eqref{liouville} gives
	\begin{equation}\label{coef}
	\sum_{i\ge0} \left(i(i-1)\xi_i+\gamma_i-\delta_i\right)r^i=0,
	\end{equation}
	with $\gamma_i=\sum_{j=0}^{i} j\xi_j\alpha_{i-j}$ and $\delta_i=\sum_{j=0}^{i} \xi_j\beta_{i-j}$. Thus all the coefficient have to be null in \eqref{coef}. By the explicit values of $\alpha_0$, $\alpha_1$, $\beta_0$, $\beta_1$ we can show by induction that, for $i\le p$, it holds 
	\[
	\xi_i(i(i+m-2)-p(p+m-2))=0,
	\]
	that is $\xi_i=0$ for $i<p$. Moreover this allows us to choose $\xi_p=1$. For $i>p$, equality \eqref{coef} reads as 
	\begin{equation}\label{coef0}
	\xi_i(i(i+m-2)-p(p+m-2)) + \sum_{j=0}^{i-1}(j\alpha_{i-j}-\beta_{i-j})\xi_j=0
	\end{equation}
	and the coefficient $\xi_i$ are recursively well defined. Moreover, equation \ref{coef0} brings to
	\begin{equation}\label{coef1}
	|\xi_i|\left(i(i+m-2)-p(p+m-2)\right) \le \sum_{j=0}^{i-1} j|\xi_j||\alpha_{i-j}|+\sum_{j=0}^{i-1} |\xi_j||\beta_{i-j}|
	\end{equation}
	 We are left to show that the sum defining $a(r)$ converges for $r\in[0,\diam(M))$, that is: the formal solution is a solution indeed.
	 
	 Fix some positive $R<\diam(M)$ and note $C_R = \sum_{i\ge0}(|\alpha_i|+|\beta_i|)R^i$. For $i\in\N$, define  $A_i=\max_{j=0,\dots,i}|\xi_j|R^j$. Then by \eqref{coef1} it follows
	\[
	\begin{aligned} 
	(i(i+m-2)-p(p+m-2))|\xi_i|R^i & \le \sum_{j=0}^{i-1} j|\xi_j|R^j|\alpha_{i-j}|R^{i-j} + \sum_{j=0}^{i-1} |\xi_j|R^{j}|\beta_{i-j}|R^{i-j} \\
		&\le A_{i-1}\left(i \sum_{j=0}^{i-1} R^{i-j}(|\alpha_{i-j}|+ |\beta_{i-j}|)\right)\\
		&\le A_{i-1}iC_R.
	\end{aligned} 
	\]
	Therefore, there exists $i_0\in\N$ such that for all $i\ge i_0$, $|\xi_i|R^i\le A_{i-1}$. This implies that $(|\xi_i|R^i)_{i\in\N}$ is a bounded sequence and the sum defining $a_k(r)$ converges for any $r<R$. Since $R$ is arbitrary, it converges for any $r\in[0,\diam(M))$.	 
\end{proof}

\begin{corollary}\label{cor-Steklov_spectrum_balls}
	For each $r>0$, the functions $(f_k(w))_{k\in\N}$ are eigenfunctions of the \linebreak Dirichlet-to-Neumann operator on $S_r$. Moreover the associated Steklov eigenvalues \linebreak $(\sigma_k(r))_{k\in\N}$ are smooth functions of $r$, and solve  the following ODE
	\begin{equation}\label{eqn-ODE_Steklov}
		\sigma_k' + \sigma_k^2 + (m-1)h\sigma_k = \lambda_k.
	\end{equation}
\end{corollary}
\begin{proof}
	Fix some $R>0$. Up to the multiplicative constant $a_k(R)$, the function $F_k$ given by
	Proposition \ref{prop-Spherical_harm} is the harmonic filling of $f_k(w)$ in the ball $B_R$.
	As the gradient of $f_k(w)$ is tangent to the sphere $S_R$ we have that the normal derivative
	of $F_k$ along the boundary of $B_R$ is
	$$
	\dd{F_k}{r} = a_k'(R)f_k(w) = \frac{a_k'(R)}{a_k(R)}F_k,
	$$
	so that $f_k(w)$ is a Steklov eigenfunction associated to the eigenvalue
	$\sigma_k(R)=\frac{a_k'(R)}{a_k(R)}$.
	
	Differentiating $\frac{a_k'(r)}{a_k(r)}$ and using (\ref{eqn-ODE_harmonic2}) we get the desired ODE.
\end{proof}

%
\subsection{The first eigenfunctions of a ball in a non-compact ROSS}

The particular case of the first spherical harmonics is easy to handle and the associated Steklov eigenfunctions can be computed explicitly. This can be done in the general setting of a harmonic manifold.

Let $M$ be a non-compact harmonic manifold with density function $\theta(t)$, and let
$h(t)=\frac{\theta'(t)}{(m-1)\theta(t)}$ be the mean curvature of spheres of radius $t$.
Let $o$ be some fixed point in $M$ and note $r=d(o,.)$ the distance function to $o$. For any $\xi\in T_oM$ consider the function
$f=\langle \xi,. \rangle$ and let 
$$
a(t)=\frac{\int_0^t \theta(s)\,ds}{\theta(t)}.
$$
This function $a$ has the following properties :
\begin{proposition}\label{prop-function_a}
	If $v(t)=|B_t(o)|$ is the volume of the ball of radius $t$, then $a(t)=\frac{v(t)}{v'(t)}$ and $a$ is an increasing function on $\R_+$.
\end{proposition}
\begin{proof}
	As $v(t)=|\S^{m-1}|\int_0^t \theta(s)\,ds$ we get $v'(t)=|\S^{m-1}|\theta(t)=|\partial B_t(o)|$ and $a(t)=\frac{v(t)}{v'(t)}$. 	A simple computation gives
	\begin{equation}\label{eqn-EDO_a}
		a'(t)=1-(m-1)h(t)a(t).
	\end{equation}
	
	As we have $\Delta r=(m-1)h(r)$, using that the function $h$ is decreasing (see Proposition \ref{prop-harm_mfd_properties}) and integrating $\Delta r$ over $B_t\setminus B_\varepsilon$ for some $0 < \varepsilon < t$ we get
	$$
	(m-1)h(t)(v(t)-v(\varepsilon)) < \int_{B_t\setminus B_\varepsilon}\Delta r dv_M\le |\partial B_t| - |\partial B_\varepsilon|
	$$
	where we used that $\nabla r$ is the outward unit normal on $\partial B_t$ and the inward unit normal on $B_\varepsilon$. Letting $\varepsilon$ tend to $0$ we get $(m-1)h(t)v(t) < v'(t)$ and
	$$
	0 < 1 - (m-1)h(t)\frac{v(t)}{v'(t)} = a'(t).
	$$
\end{proof}

As stated in Proposition \ref{prop-harm_mfd_properties}, the function $h(r)$ decreases to some constant $h_0$ which is the mean curvature of horospheres of $M$. If $h_0=0$ then $M$ is the euclidean space and $a(r)=r$. Otherwise, using that $\theta'(r)=(m-1)h(r)\theta(r)$, we get that $\lim_{r\to\infty}a(r)=\frac{1}{(m-1)h_0}>0$.

\begin{proposition}\label{prop-Harm_funct_harm_mfd}
	The function $f(w)$ is an eigenfunction of the Laplacian of any geodesic sphere $S_r(o)$
	associated to the eigenvalue $-(m-1)h'(r)$ and the function $F=a(r)f(w)$ is harmonic.
	In particular, $F$ is a Steklov eigenfunction of the ball $B_r(o)$ associated to the
	eigenvalue $\frac{a'(r)}{a(r)}$.
\end{proposition}
\begin{proof}
	Cf. \cite{Ranjan-Shah, Knieper-Peyerimhoff}.
\end{proof}

The main question is whether the Steklov eigenvalue associated to the first spherical harmonics is
the first non-zero eigenvalue. We prove this is the case for a non-compact ROSS.
The first step is a comparison result for solutions of equation (\ref{eqn-ODE_Steklov}) :
\begin{lemma}
	Let $\mu_1,\mu_2:\R_+^*\to\R$ be two functions such that $ \mu_1 \le \mu_2$ on $\R_+^*$,
	and, for $k=1,2$, let $\alpha_k$ be non-negative solutions of
	$$
	\alpha_k' + \alpha_k^2 + (m-1)h\alpha_k = \mu_k.
	$$
	If there exists $t_0>0$ such that $\alpha_1\le\alpha_2$ on $]0,t_0]$ then we have that
	$\alpha_1 \le \alpha_2$ on $\R_+^*$.
\end{lemma}
\begin{proof}
	Let $\beta=\alpha_2-\alpha_1$. From the differential equations satisfied by $\alpha_1$ and $\alpha_2$ we have
	$$
	\beta'+b\beta = \mu_2 - \mu_1 \ge 0
	$$
	where $b=\alpha_1+\alpha_2+(m-1)h$. If $\gamma$ is a positive solution of $\gamma'+b\gamma=0$, we get $\beta'-\frac{\gamma'}{\gamma}\beta \ge 0$.	As $\gamma$ is positive we also have $\left(\frac{\beta}{\gamma}\right)' \ge 0$ which implies that, for all $t\ge t_0$, $\beta(t)\ge\frac{\beta(t_0)}{\gamma(t_0)}\gamma(t)\ge 0$
\end{proof}

From this lemma, it is sufficient to prove that, for any $r>0$, $\lambda_1(r)$ is the first non-zero eigenvalue of the geodesic sphere $S_r(o)$ to get that $\sigma_1(r)$ is the first non-zero Steklov eigenvalue of the geodesic ball $B_r(o)$. Up to some rescaling, the spectrum of the Berger metric $g_c-\sum_{k=1}^{d-1}\delta_k\otimes\delta_k+t\sum_{k=1}^{d-1}\delta_k\otimes\delta_k$ was computed in \cite{Tanno1,Tanno2}. It appears that if the parameter $t$ is greater than $1$, than the first spherical harmonics remain associated to the first eigenvalue. For the geodesic sphere $S_r(o)$, the parameter is $t=\cosh(r)>1$, so that we have :
\begin{proposition}
	Let $M$ be a non-compact ROSS and $o\in M$ be some fixed point. For an orthonormal
	basis $(\xi_1,\dots,\xi_m)$ of $T_oM$, consider the functions $f_i=\langle \xi_i,.\rangle$, $i=1,\dots,m$.
	
	For any $R>0$, the functions $f_i(w)$ are eigenfunctions of the Dirichlet-to-Neumann operator on $S_R(o)$ associated to the first Steklov eigenvalue $\sigma_1(B_R)$, and their harmonic filling in $B_R(o)$ are $F_i=\frac{a(r)}{a(R)}f_i(w)$.
\end{proposition}

	The above Proposition has to be compared with Theorem 2.1 in \cite{Binoy-Santhanam}. Note that such a statement does not hold on compact ROSS. In this case, the parameter $t$ of the Berger metric is $t=\cos(r)<1$ and for a geodesic sphere $S_r(o)$ with $r$ large enough, the functions $f_i$ are no more associated to the first non-zero eigenvalue of the Laplacian of $S_r(o)$ (see the computations in \cite{Tanno1,Tanno2} and Remark 7.3 in \cite{Berard-Bergery-Bourguignon}).

%
%
%
%
\section{Brock-Weinstock inequalities in non-compact ROSS}\label{sec:steklov}

In this section we prove Theorem \ref{th:main1}.
%
\subsection{A weighted isoperimetric inequality on Harmonic manifolds}

Let $M$ be a non-compact harmonic manifold of dimension $m$. This subsection is devoted to the proof of a weighted isoperimetric inequality which is a key point in the eigenvalue estimate. We first introduce the weighted perimeter involved in this inequality.
\begin{definition}[Weighted perimeter]
	Let $M$ be a non-compact harmonic manifold with volume density function $\theta$, and let $o\in M$ be some fixed point.
	The weighted perimeter of a domain $\Omega\subset M$ is
	\[
	P_o (\Omega)=\int_{\partial\Omega}a(r)^2\,dv_{\partial\Omega}
	\]
	where $r$ is the distance function to $o$ and $a(r)=\frac{\int_0^r\theta(s)ds}{\theta(r)}=\frac{v(r)}{v'(r)}$.
\end{definition}

	Isoperimetric inequalities for weighted perimeters were considered in \cite{B-B-M-P} in the Euclidean space. For the perimeter $P_o$ defined above, we prove such an inequality on harmonic manifolds using a calibration technique as in \cite{Brasco-DePhilippis-Ruffini}. In the special case of ROSS, the following inequality was proved in lemma 3.3 of \cite{Binoy-Santhanam}.

\begin{theorem}\label{isoperimetriqueponderee}
	Let  $M$ be a non-compact harmonic manifold and $o\in M$ be some fixed point. For any domain $\Omega\subset M$ we have,
	\begin{equation}\label{isoperimetric}
	P_o (\Omega)\ge P_o (B),
	\end{equation}
	where $B$ is the ball centered in $o$ such that $|\Omega|=|B|$. Moreover equality holds if and only if $\Omega=B$.
\end{theorem}
\begin{proof}
The proof relies on a calibration technique developed in \cite{Brasco-DePhilippis-Ruffini} (see also \cite{Kolesnikov-Zhdanov}).
Let   
\[
G(r)=\div\left( a(r)^2\frac{\partial }{\partial r}\right) =a(r)^2\Delta r + 2a(r)a'(r)=(m-1)a^2(r)h(r)+2a(r)a'(r).	
\]
From \eqref{eqn-EDO_a} we get $G(r)=a(r)+a(r)a'(r)$, and a direct computation shows then that $G'$ is given by
\[
G'(r)=2a'(r)^2-(m-1)h'(r)a(r)^2.
\]
Since $r\mapsto h(r)$ is a decreasing function (cf. Proposition \ref{prop-harm_mfd_properties}), we get  that $G$ is an increasing function of the distance from $o$:
\begin{equation}\label{G'>0}
G'(r)>0. 
\end{equation}
Let $B$ be the ball centered at $o$ such that $|B|=|\Omega|$, and let $R$ be its radius. We have
\[
\int_{\Omega} G(r)\,dv_M=\int_{\Omega} \div\left(a(r)^2\frac{\partial}{\partial r}\right)\,dv_M=\int_{\partial\Omega}a(r)^2\langle \frac{\partial}{\partial r},\nu_\Omega \rangle\,dv_{\partial\Omega}\le P_o (\Omega),
\]
and
\[
\int_{B} G(r)\,dv_M=\int_{B} \div\left(a(r)^2\frac{\partial}{\partial r}\right)\,dv_M=\int_{\partial B}a(R)^2\langle \frac{\partial}{\partial r},\nu_{B} \rangle\,dv_{\partial B}= P_o (B).
\]
Thus
\begin{equation}\label{isoper}
\begin{aligned}
P_o (\Omega)-P_o (B)&\ge \int_{\Omega}G(r)\,dv_M-\int_{B}G(r)\,dv_M\\
&=\int_{\Omega\setminus B}G(r)\,dv_M-\int_{B\setminus\Omega}G(r)\,dv_M\\
&\ge \int_{\Omega\setminus B}|G(r)-G(R)|\,dv_M +\int_{B\setminus\Omega}|G(r)-G(R)|\,dv_M\\
\end{aligned} 
\end{equation}
where in the last inequality we used the facts that $G$ is an increasing function of $r$ and that $|\Omega\setminus B|=|B\setminus\Omega|$ (which follows from $|\Omega|=|B|$).
To characterize the equality case, we notice that the last quantity is positive when $\Omega\not=B$, as $|G(r)-G(R)|>0$ for $r\not=R$. Thus we can have equality in  \eqref{isoperimetric} only if $\Omega=B$.
\end{proof}

%
\subsection{Quantitative stability of the weighted isoperimetric inequality}

In this section we show a quantitative version of the isoperimetric inequality proved in the previous section. In what follows we work in the framework of the previous section: $B$ is the ball centered in the fixed point $o$ such that $|B|=|\Omega|$, $R$ is its radius, for each $s\ge 0$ we note $B_s$ the ball of radius $s$ centered in $o$, and we denote 
\[
\delta=|\Omega\setminus B|=|B\setminus\Omega|.
\] 
As before, we set for $s\ge0$, $v(s)=|B_s|$ the volume of the ball of radius $s$; as the manifold $M$ is harmonic, the volume of $B_s$ does not depend on the center of the ball, and since $M$ is non-compact, $v$ is an increasing function on $\R_+$. Let moreover $R_{ext}>R_{int}\ge0$ be defined by the relations
\[
|B_{R_{ext}}|-|B_R|=|\Omega\setminus B|=|B\setminus\Omega|=|B_R|-|B_{R_{int}}|.
\]
\begin{lemma}
	It holds
	\begin{itemize}
		\item[$(i)$] $R_{ext}=v^{-1}(|\Omega|+\delta)$;
		\item[$(ii)$] $R_{int}=v^{-1}(|\Omega|-\delta)$;
		\item[$(iii)$] $R=v^{-1}(|\Omega|)$;
		\item[$(iv)$] $a(t)=\frac{v(t)}{v'(t)}$;
		\item[$(v)$] $a(v^{-1}(s))=s(v^{-1})'(s)$.
	\end{itemize}
\end{lemma}
\begin{proof}
The proof of $(i)-(iii)$ follows straightforwardly from the definitions of $R_{ext}$, $R$, $R_{int}$ and $\delta$. The point $(iv)$ was already observed in Proposition \ref{prop-function_a} as a consequence of the definitions of the density function and of $a$ and $v$. To prove $(v)$ just use $(iv)$ to get
\[
(v^{-1})'(s)=
\frac{1}{v'(v^{-1}(s))}\overset{(iv)}{=}\frac{a(v^{-1}(s))}{s}.
\]
\end{proof}

The quantitative stability for the weighted isoperimetric inequality is obtained in two step: we first prove it for domains which are a priori close to the ball and then use a continuity argument to get it for arbitrary domains.

\begin{proposition}\label{coroquantitative}
	There exists $\bar{\delta}=\bar{\delta}(|\Omega|,M)>0$ and $C=C(|\Omega|,M)>0$ such that if $|\Omega\setminus B|<\bar{\delta}$ then
	\[
	P_o (\Omega)-P_o (B)\ge C|\Omega\setminus B|^2.
	\]
\end{proposition}
\begin{proof}
	By \eqref{isoper} we know that 
	\[
	P_o (\Omega)-P_o (B)\ge\int_{\Omega\setminus B} G(r)\,dv_M-\int_{B\setminus \Omega} G(r)\,dv_M.
	\]
	Since $G$ is an increasing function it is easy to show that
	\begin{equation}\label{ineq:lem2}
	\begin{aligned}
	P_o (\Omega)-P_o (B)&\ge\int_{\Omega\setminus B} G(r)\,dv_M-\int_{B\setminus \Omega} G(r)\,dv_M\\
	&\ge \int_{B_{R_{ext}}\setminus B} G(r)\,dv_M-\int_{B\setminus B_{R_{int}}} G(r)\,dv_M
	\\
	&=|\mathbb{S}^{m-1}|\int_{R}^{R_{ext}}G(t)\theta(t)\,dt-|\mathbb S^{m-1}|\int_{R_{int}}^{R}G(t)\theta(t)\,dt.
	\end{aligned}
	\end{equation}
	Let us recall now that $G(t)=a(t)+a(t)a'(t)$, from which we get
	$$
	G(t)\theta(t)=a(t)\theta(t)+a'(t)\int_0^t\theta(s)\,ds = \frac{d}{dt}\left[a(t)\int_0^t\theta(s)\,ds\right]
	$$
	Thus we have, for $0<S<T$,
	\[
	\int_S^T G(t)\theta(t)\,dt = a(T)\int_0^T\theta(t)\,dt-a(T)\int_0^S\theta(t)\,dt,
	\]
	By applying this equality with $S=R$, $T=R_{ext}$ and $S=R_{int}$, $T=R$ in the last term of \eqref{ineq:lem2} we get
	\begin{equation}\label{eqQ1}
	P_o (\Omega)-P_o (B)\ge a(R_{ext})|B_{R_{ext}}|-2a(R)|B|+a(R_{int})|B_{R_{int}}|.
	\end{equation}
	Let $g:\R_+\to\R$ be defined by $g(s)=s^2(v^{-1})'(s)=sa(v^{-1}(s))$ so that \eqref{eqQ1} can be written as
	\begin{equation}\label{eqQ2}
	P_o (\Omega)-P_o (B) \ge g(|\Omega|+\delta)-2g(|\Omega|)+g(|\Omega|-\delta)=\frac{\delta^2}{2}(g''(s_1)+g''(s_2)).
	\end{equation}
	for some $s_1,s_2\in]|\Omega|-\delta,|\Omega|+\delta[$. A simple computation now gives
	$$
	g''(s)=\frac{2a'(t)^2-(m-1)a(t)^2h'(t)}{v'(t)} > 0,
	$$
	where $t=v^{-1}(s)$. Choosing $\bar{\delta}$ such that $g''(s)\ge\frac{g''(|\Omega|)}{2}$ on $]|\Omega|-\bar{\delta},|\Omega|+\bar{\delta}[$, and assuming that $|\Omega\setminus B|=\delta<\bar{\delta}$, we get
	\[
	P_o (\Omega)-P_o (B)\ge  \frac{g''(|\Omega|)}{2}|\Omega\setminus B|^2.
	\]
\end{proof}

\begin{remark}
	Following \cite{Brasco-DePhilippis-Ruffini}, it is possible to get the full stability result from equation (\ref{eqQ2}) in the special case where the function $g''$ is non-increasing. In fact, as $\delta=|\Omega\setminus B|\le|\Omega|$, equation (\ref{eqQ2}) gives
	$$
	P_o (\Omega)-P_o (B) \ge g''(2|\Omega|)|\Omega\setminus B|^2
	$$
	without any restriction on $\delta$. A computation of the third derivative of $g$ gives
	$$
	g'''(s)=\frac{6a'(t)a''(t) + (m-1)^2a(t)^2h(t)h'(t) - (m-1)a(t)^2h''(t)}{v'(t)^2}.
	$$
	where $t=v^{-1}(s)$. The sign of $g'''$ is not obvious for an arbitrary harmonic manifold. However, if $M$ is a non-compact ROSS, a direct computation proves that $a''\le 0$ and $h''\ge 0$ so that $g''$ is non-increasing and the stability result follows.
\end{remark}

To get a full stability result in the general case, we have to show that it is not restrictive to only consider sets which are near to the ball in $L^1$. Namely we need the following continuity lemma.
\begin{lemma}\label{continuity}
	Let $(\Omega_k)_{k\in\N}$ be a sequence of domains in $M$ such that, for all $k\in\N$, $|\Omega_k|=c$ and let $B$ be the ball centered in $o$ with $|B|=c$. Suppose that  $P_o (\Omega_k)-P_o (B)\to0$ as $k\to\infty$. Then 
	\[
	\lim_{k\to\infty}|\Omega_k\setminus B|\to0.
	\]
\end{lemma}
\begin{proof}
	Let $R$ be the radius of $B$. For any $\rho>R$, by repeating the proof of Theorem \ref{isoperimetriqueponderee}, until formula \eqref{isoper}, we get that
	\[
	\begin{aligned}
	P_o(\Omega_k)-P_o(B)&\ge \int_{\Omega_k\setminus B}|G(r)-G(R)|\,dv_M\\
	&\ge \int_{\Omega_k\setminus B_\rho}(G(r)-G(R))\,dv_M\\
	&\ge (G(\rho)-G(R))|\Omega_k\setminus B_\rho| \\
	&\ge (G(\rho)-G(R))|\Omega_k\setminus B| - (G(\rho)-G(R))|B_\rho\setminus B|
	\end{aligned}
	\]
	Therefore, we have
	$$
	|\Omega_k\setminus B| \le \frac{P_o(\Omega_k)-P_o(B)}{G(\rho)-G(R)} + |B_\rho\setminus B|
	$$
	and when $k$ tends to $\infty$ we get, for any $\rho>R$,
	\[
	\limsup_{k\to \infty}|\Omega_k\setminus B|\le|B_\rho\setminus B|.
	\] Letting $\rho$ tend to $R$ gives the result.
\end{proof}

\begin{theorem}\label{thm-quantitative_isop}
	Let  $M$ be a non-compact harmonic manifold and $o\in M$ be some fixed point. Let $\Omega\subset M$ be a domain and $B$ be the ball centered in $o$ such that $|B|=|\Omega|$. Then there exists a constant $C=C(|\Omega|,M)$ such that 
	\begin{equation}\label{eq:quantitative}
	P_o (\Omega)-P_o (B)\ge C|\Omega\setminus B|^2.
	\end{equation}
\end{theorem}
\begin{proof}
	Let $\Omega\in M$ and let $\bar\delta>0$ be the parameter of Proposition \ref{coroquantitative}. If $|\Omega\setminus B|<\bar{\delta}$ we are done. Otherwise, by Lemma \ref{continuity} there exists $\bar{\varepsilon}>0$ such that 
	\[
	P_o (\Omega)-P_o (B)>\bar{\varepsilon}=\frac{\bar{\varepsilon}}{|\Omega\setminus B|^2}|\Omega\setminus B|^2\ge \frac{\overline{\varepsilon}}{|\Omega|^2}|\Omega\setminus B|^2.
	\]  
\end{proof}

%
\subsection{Choosing a base point}

Let $M$ be a harmonic manifold and $\Omega$ be a bounded domain of $M$. The proof of the Brock-Weinstock inequality relies on a transplantation method consisting in plugging the first Steklov eigenfunctions of a ball in the Rayleigh quotient of the domain. In order to get an estimate of the first eigenvalue of $\Omega$ we need these functions to be
orthogonal to the constant. The aim of the following lemma is to prove that we can chose the center of the ball in such a way.

For $y\in M$, note $r_y=d(y,.)$ the distance function to $y$, and for $\xi\in U_yM$, let $F_{y,\xi}$ be the
harmonic function on $M$ given by Proposition \ref{prop-Harm_funct_harm_mfd}:
$$
\forall x\in M\ \ \ \ \ F_{y,\xi}(x) = a(r_y)\langle \xi,w_y(x) \rangle
$$
where $w_y(x)$ is the unique unit  vector in $U_yM$ such that $x=\exp_y(r_y(x)w_y(x))$.
\begin{lemma} \label{lem-Base_point}
	Let $M$ be a non-compact harmonic manifold and $\Omega\subset M$ a compact set. There exists a  point $o\in M$ such that 
	\begin{equation}\label{barycenter}
	\forall \xi\in T_oM\ \ \ \int_{\partial\Omega}F_{o,\xi}(x)\,dv_{\partial\Omega}(x)=0.
	\end{equation} 
\end{lemma}
\begin{proof}
	Consider the functions $b:\R_+\to\R_+$ and $B:M\to\R_+$ defined by 
	\[
	b(r)=\int_{0}^r a(s)\,ds,\qquad B(y)=\int_{\partial\Omega}b(d(x,y))\,dv_{\partial\Omega}(x).
	\]
	For any $\xi\in T_yM$ we have
	\[
	\begin{aligned}
	\langle \nabla B(y),\xi\rangle&=\int_{\partial\Omega } b'(d(x,y))\langle \nabla r_x(y),\xi\rangle \,dv_{\partial\Omega}(x)\\
	&=-\int_{\partial\Omega }a(d(x,y))\langle w_y(x),\xi \rangle \,dv_{\partial\Omega}(x)\\
	&=-\int_{\partial\Omega } F_{y,\xi}(x) \,dv_{\partial\Omega}(x).
	\end{aligned}
	\]
	Suppose now that $o$ is a minimum of $B$. Since  $\nabla B(o)=0$ we get that $o$ is such that \eqref{barycenter} holds true.
	
	To show that such a point $o$ exists indeed, we have just to notice that $\lim_{r\to\infty}b(r)=+\infty$, as $a$ is an increasing function (see Proposition \ref{prop-function_a}). Therefore the sublevels of $B$ are compact sets and $B$ has a minimum.
\end{proof}

\begin{remark}
	It is a natural question whether the minimum of $B$ is unique. As the function $a$ is increasing (cf. Proposition \ref{prop-function_a}), the function $b$ is convex on $\R_+$. If $M$ is a non-compact ROSS or a Dameck-Ricci space, the curvature being non-positive, the distance function is also convex. From these two facts we have that $B$ is convex and thus has a unique minimum.
\end{remark}

%
\subsection{The Brock-Weinstock inequality holds true on non-compact ROSS}

From now on we assume that $M$ is a non-compact ROSS and we use the notations of Section~\ref{sec:ROSS}. Let $o\in M$ be a fixed point. We first make the connection between the weighted perimeter and the first Steklov eigenfunctions of a ball. Following the notation of section 3, for any $x\in M$ we note $r(x)=d(o,x)$ and $w(x)\in U_oM$ the unique unit vector such that $x=\exp_o(r(x)w(x))$. From now on we fix an orthonormal basis $(\xi_1,\dots,\xi_m)$ of $T_oM$ and we write, for $i=1,\dots,m$, $F_i=a(r)\langle \xi_i,w \rangle$.
\begin{lemma}
	We have
	\[
	P_o (\Omega)=\sum_{i=1}^m \int_{\partial \Omega}F_i^2\,dv_{\partial\Omega}.
	\]
\end{lemma}
\begin{proof}
	By the definition of $F_i$ we have
	\[
	\begin{aligned} 
	\sum_{i=1}^m \int_{\partial \Omega}F_i^2\,dv_{\partial\Omega}
	&=\int_{\partial \Omega}a^2(r)\sum_{i=1}^m\langle\xi_i,w(x)\rangle^2\,dv_{\partial\Omega} \\
	&=\int_{\partial\Omega} a^2(r)\,dv_{\partial\Omega}
	\\
	&=P_o (\Omega).
	\end{aligned}
	\]
\end{proof}

The first eigenvalue of the ball $B_R$ satisfies
$$
\sigma_1(B_R)\int_{\partial B_R}F_i^2dv_{\partial B_R} = \int_{B_R}|\nabla F_i|^2dv_M
$$
and summing over $i$ we get
$$
P_o (B_R)\sigma_1(B_R) = \int_{B_R}\sum_{i=1}^{m}|\nabla F_i|^2dv_M.
$$
The second observation is that the integrand of the right-hand term is a radial function.
\begin{lemma}
	For any $x\in M$ we have
	$$
	\sum_{i=1}^{m}|\nabla F_i(x)|^2 = H(r(x))
	$$
	where $H:\R_+\to\R_+$ is given by $H(r)=a'(r)^2-(m-1)h'(r)a(r)^2$.
\end{lemma}
\begin{proof}
	For $i=1,\dots,m$, note $f_i=\langle \xi_i,w \rangle$, so that $F_i=a(r)f_i$. Let $x\in M$ and note $\eta=w(x)$, so that, from \cite[Section 7]{Knieper-Peyerimhoff}  and the computation of Jacobi tensors on $M$ (cf. section \ref{sec:ROSS}) we have
	$$
	\begin{aligned}
		\nabla f_i(x) & = A_\eta^*(r)^{-1}(\xi_i-\langle \xi_i,\eta \rangle\eta) \\
			& = \sum_{j=2}^m \langle \xi_i,\eta_j \rangle A_\eta^*(r)^{-1}(\eta_j) \\
			& = \sum_{j=2}^d \frac{\langle \xi_i,\eta_j \rangle}{\alpha(r(x))}\eta_j
				+ \sum_{j=d+1}^m \frac{\langle \xi_i,\eta_j \rangle}{\beta(r(x))}\eta_j
	\end{aligned}
	$$
	where $(\eta_1,\dots,\eta_m)$ is an orthonormal basis of $T_oM$ with $\eta_1=\eta$ and $\eta_{j+1}=J_j\eta$ for $j=1,\dots,d-1$. From these equalities we get
	$$
	\begin{aligned}
		\sum_{i=1}^{m}|\nabla f_i(x)|^2 &
			= \sum_{i=1}^{m}\sum_{j=2}^d \frac{\langle \xi_i,\eta_j \rangle^2}{\alpha(r(x))^2}
			+ \sum_{i=1}^{m}\sum_{j=d+1}^m \frac{\langle \xi_i,\eta_j \rangle^2}{\beta(r(x))^2} \\
			& = \frac{d-1}{\alpha(r(x))^2} + \frac{m-d}{\beta(r(x))^2}
	\end{aligned}
	$$
	Therefore, the function $\sum_{i=1}^{m}|\nabla f_i|^2$ is a radial function on $M$, and since each of the $f_i$ are eigenfunctions of geodesic spheres with eigenvalue $-(m-1)h'(r)$, taking the mean over $S_r(o)$ we get
	$$
	\begin{aligned}
		\sum_{i=1}^{m}|\nabla f_i|^2 &
			= \frac{1}{|S_r(o)|}\sum_{i=1}^{m}\int_{S_r(o)}|\nabla f_i|^2dv_{S_r(o)} \\
			& = \frac{-(m-1)h'(r)}{|S_r(o)|}\sum_{i=1}^{m}\int_{S_r(o)}f_i^2dv_{S_r(o)} \\
			& = -(m-1)h'(r)
	\end{aligned}
	$$
	
	For $i=1,\dots,m$, we have $\nabla F_i=a'(r)f_i\dd{}{r}+a(r)\nabla f_i$, and, since
	$\dd{}{r}$ and $\nabla f_i$ are orthogonal we obtain
		$$
		\begin{aligned}
			\sum_{i=1}^{m}|\nabla F_i|^2 &
			= a'(r)^2\sum_{i=1}^{m}f_i^2 + a(r)^2\sum_{i=1}^{m}|\nabla f_i|^2 \\
			& = a'(r)^2-(m-1)h'(r)a(r)^2.
		\end{aligned}
		$$
\end{proof}
In the sequel, we note $Q(\Omega)=\int_{\Omega}H(r)$, so that $\sigma_1(B_R)=\frac{Q(B_R)}{P_o (B_R)}$.

\begin{lemma}
	The function $H$ is non-increasing on $\R_+$.
\end{lemma}
\begin{proof}
	We have
	$$
	H'(r)=2a'(r)a''(r) - (m-1)h''(r)a(r)^2 - 2(m-1)h'(r)a(r)a'(r),
	$$
	and since $a'(r)=1-(m-1)a(r)h(r)$ we get
	$$
	H'(r)=-2(m-1)a'(r)^2h(r) - 4(m-1)a(r)a'(r)h'(r) - (m-1)h''(r)a(r)^2.
	$$
	From the expression \ref{eqn-mean_curvature} of $h(r)$, and using that $\frac{d-1}{m-1}\le 1$, we get
	$$
	h(r)h''(r) \ge 2\left(\frac{1}{s(r)^2}+\frac{d-1}{m-1}\frac{1}{c(r)^2}\right)^2
		\ge \left(\frac{1}{s(r)^2}-\frac{d-1}{m-1}\frac{1}{c(r)^2}\right)^2 = h'(r)^2
	$$
	so that
	$$
	\begin{aligned}
		H'(r)  &
			\le -2(m-1)a'(r)^2h(r) - 4(m-1)a(r)a'(r)h'(r) - (m-1)a(r)^2\frac{h'(r)^2}{h(r)} \\
			& = -2(m-1)h(r)\left(a'(r) + a(r)\frac{h'(r)}{h(r)}\right)^2 \\
			& \le 0.
	\end{aligned}
	$$
\end{proof}

\begin{theorem}\label{steklovonH}
	Let $M$ be a non-compact ROSS, let $\Omega\subset M$ be a domain and let $B$ be a ball such that $|\Omega|=|B|$. Then
	\[
	\sigma_1(\Omega)\le\sigma_1(B),
	\]
	with equality if and only if $\Omega$ is a ball.
\end{theorem}
\begin{proof}
	Using Lemma \ref{lem-Base_point} we choose a base point $o\in M$ such that the functions $F_i$ are orthogonal to the constant function on $\partial\Omega$. Therefore, for each $i=1,\dots,m$ we have
	$$
	\sigma_1(\Omega)\le\frac{\int_\Omega|\nabla F_i|^2\,dv_M}{\int_{\partial\Omega}F_i^2\,dv_{\partial\Omega}}
	$$
	Taking a sum over $i$ we get
	$$
	\sigma_1(\Omega)P_o (\Omega)=\sigma_1(\Omega)\sum_{i=1}^{m}\int_{\partial\Omega}F_i^2\,dv_{\partial\Omega}
		\le \sum_{i=1}^{m}\int_\Omega|\nabla F_i|^2\,dv_M = Q(\Omega)
	$$
	By the definition of $Q$ and since $H$ is non-increasing we have
	\[
	\begin{aligned}
	Q(\Omega)- Q(B_R)	&=\int_{\Omega} H(r)\,dv_M-\int_{B_R}H(r)\,dv_M\\
	&=\int_{\Omega\setminus B_R} H(r)\,dv_M-\int_{B_R\setminus\Omega}H(r)\,dv_M\\
	&\le \int_{\Omega\setminus B_R} \left(H(R)-H(r)\right)\,dv_M\le 0.
	\end{aligned} 
	\]
	Therefore $Q(\Omega)\le Q(B_R)$ and, using Theorem \ref{isoperimetriqueponderee} we get
	\begin{equation}
	\label{eq:steklov}
	\sigma_1(\Omega)P_o (B_R) \le \sigma_1(\Omega)P_o (\Omega) \le Q(\Omega)
		\le Q(B_R)
	\end{equation} 
	and
	$$
	\sigma_1(\Omega) \le \frac{Q(B_R)}{P_o (B_R)} = \sigma_1(B_R).
	$$
	Moreover, in case of equality we have $P_o (\Omega)=P_o (B_R)$ which implies that
	$\Omega=B_R$.
\end{proof}
By combining \eqref{eq:quantitative} and \eqref{eq:steklov} it is easy to produce the stability version of our quantitative inequality:
\begin{theorem}
	Let $M$ be a non-compact ROSS. For any $v>0$, there exists a positive constant $C=C(M,v)$ such that, for any domain $\Omega\subset M$ with $|\Omega|=v$ we have
	\[
	\sigma_1(\Omega)\left(1+C|\Omega\setminus B|^2\right)\le\sigma_1(B),
	\]
	where $B$ is a geodesic ball with $|B|=|\Omega|$
\end{theorem}
{
\begin{proof}
	Let $o$ be the base point given by Lemma \ref{lem-Base_point} and $B_R$ the geodesic ball centered in $o$ whose volume is equal to $|\Omega|$. Because the ROSS are harmonic manifolds, the weighted perimeter of the ball $B_R$ does not depend on the base point, and its radius only depends on its volume. Therefore, the quantity $P_o(B_R)$ only depends on the volume $|\Omega|$.
	
	From inequalities  \eqref{eq:quantitative} and \eqref{eq:steklov} we get
	\[
	\sigma_1(\Omega)\Bigl(1+\frac{C}{P_o(B_R)}|\Omega\setminus B_R|^2\Bigr) \le \sigma_1(B_R)
	\]
	where the constant $C$ is given by Theorem \ref{thm-quantitative_isop}, and thus only depends on $|\Omega|$.
\end{proof}
}
\begin{remark}
	Note that in general we can not freely state that, for each $i=1,\dots,m$,
	\[
	\int_{\Omega} |\nabla F_i|^2\,dv_M\le \int_{B_R} |\nabla F_i|^2\,dv_M. 
	\]	
	If this were the case then by an argument similar to that in the above  proof,  we would get the stronger inequality 
	\[
	\sum_{i=1}^n\frac{1}{\sigma_i(\Omega)}\ge\frac{n}{\sigma_i(B_R)}.
	\]
	Compare with \cite[Theorem $5.1$]{Brasco-DePhilippis-Ruffini}. 
\end{remark}

%
\section{The Brock-Weinstock inequality does not hold on $\mathbb S^m$} \label{sec:nottheball}

In this section we denote by $\sigma_k(A)$ the $k-$th Steklov eigenvalue on the sphere $\mathbb S^m$ of a domain $A\subset\mathbb S^m$, and we show that there exists a symmetrical strip $\Omega\subset \mathbb S^m$ such that $\sigma_1(\Omega)>\sigma_1(B)$ where $B$ is a ball (in the sphere) with the same volume as $\Omega$. In particular, the ball does not maximize the first Steklov eigenvalue on the sphere, that is, the Brock-Weinstock inequality does not hold true. The choice of a  spherical strip is due to the fact that in this case we are able  to compute explicitly its spectrum (cf. \cite[example 4.2.5]{Girouard-Polterovich} for a similar calculation for annulus in Euclidean spaces).

\begin{theorem}
	Let $M=\mathbb S^m$. There exists  $R>\frac{\pi}{2}$ such that $\sigma_1(\Omega_R)>\sigma_1(B_S)$, where $\Omega_R$ is the intersection of two geodesic balls of radius $R$ with antipodal centers, and $B_S$ is a geodesic ball such that $|\Omega_R|=|B_S|$.
\end{theorem}
\begin{proof}
Let $o^+$ and $o^-$ be two antipodal points on the unit sphere. Viewing $\S^m$ as a submanifold of $\R^{m+1}$ we have $T_{o^+}\S^m=T_{o^-}\S^m$ as subspaces of $\R^{m+1}$.  For $x\in\S^m$ we note $r^+(x)=d(o^+,x)$, $r^-(x)=d(o^-,x)$, and $w(x)$ the unique unit vector of $T_{o^+}\S^m=T_{o^-}\S^m$ pointing to $x$ from $o^+$ and $o^-$:
$$
x = \exp_{o^+}(r^+(x)w(x)) = \exp_{o^-}(r^-(x)w(x)).
$$

Let $R>\frac\pi2$ and $\Omega_R=B_R(o^+)\cap B_R(o^-)$, $S_R^+=\partial B_R(o^+)$, $S_R^-=\partial B_R(o^-)$ (so that $\partial\Omega_R=S_R^+\cup S_R^-$). Let $(f_k)_{k\in\N}$ be a basis of spherical harmonics of $L^2(\mathbb S^{m-1})$ where $\mathbb S^{m-1}$ is the unit sphere of $T_{o^\pm}\mathbb S^{m}$ and the ordering is such that the corresponding sequence of eigenvalues is non-decreasing. By Proposition \ref{prop-Spherical_harm} and Corollary \ref{cor-Steklov_spectrum_balls} we know that, for each $k\in\N$, there exist radial functions $a_k(r^+)$ and $a_k(r^-)$ such that
\[
F_k^\pm:=a_k(r^\pm)f_k(w),\quad k\in\mathbb N
\]
form a basis of Steklov eigenfunctions for $L^2(\partial B_R(o^\pm))$. Clearly we have that the $F_k$'s are harmonic in $\Omega_R$. Moreover, a direct computation shows that, on $\partial\Omega_R$,
\[
\frac{\partial(F_k^- \pm F_k^+)}{\partial\nu}=C_k^\pm(R)(F_k^- \pm F_k^+)
\]
with
\[
C_k^+(R)=\frac{a_k'(R)-a_k'(\pi-R)}{a_k(R)+a_k(\pi-R)},\quad C_k^-(R)=\frac{a_k'(R)+a_k'(\pi-R)}{a_k(R)-a_k(\pi-R)}.
\]
Thus $C_k^\pm(R)$ are Steklov eigenvalues of $\Omega_R$. Moreover, they constitute the whole Steklov spectrum. Indeed, as $\partial\Omega_R$ is the disjoint union of $S_R^+$ and $S_R^-$ we get an Hilbert basis of $L^2(\partial\Omega_R)$ by taking $\{f_k^+\ |\ k\in\N\}\cup\{f_k^-\ |\ k\in\N\}$ where $f_k^+=f_k(w)_{|_{S_R^+}}$ and $f_k^-=f_k(w)_{|_{S_R^-}}$. As we have
$$
(F_k^+)_{|_{\partial\Omega_R}} = a_k(R)f_k^+ + a_k(\pi-R)f_k^-\ \ \mbox{ and }\ \ 
(F_k^-)_{|_{\partial\Omega_R}} = a_k(\pi-R)f_k^+ + a_k(R)f_k^-,
$$
and since $0<a_k(\pi-R)<a_k(R)$ we easily get that $\{(F_k^++F_k^-)_{|_{\partial\Omega_R}} \ |\ k\in\N\}\cup\{(F_k^+-F_k^-)_{|_{\partial\Omega_R}} \ |\ k\in\N\}$ is an Hilbert basis of $L^2(\partial\Omega_R)$. Therefore, the Steklov spectrum of $\Omega_R$ is $\{C_k^+(R)\ |\ k\in\N\}\cup\{C_k^-(R)\ |\ k\in\N\}$.

We divide the rest of the proof into three parts: the first two aimed to prove that $\sigma_1(\Omega_R)=C_1^+(R)$ and the last one for comparing it to the first eigenvalue of the ball.

%
%
\medskip
\noindent \textbf{Claim 1:} \textit{for all $\frac{\pi}{2}<R<\pi$ and all $k\in\N$ we have $C_k^+(R)\le C_k^-(R)$.}

In fact, we will prove that, for any $k\in\N$, there exists $p\in\N$ such that $a_k(r) = r^p + \sum_{i\ge p+1}\xi_ir^i$ with $\xi_i\ge 0$ for all $i\ge p+1$.

Let $k\in\N$ and let $p\in\N$ be such that the eigenvalue of $\S^{m-1}$ associated to $f_k$ is $\lambda_{c,k}=p(p+m-2)$. Following the proof of Proposition \ref{prop-Spherical_harm} we have $a_k(r)=\sum_{i\ge0} \xi_i r^i$,
with the coefficients $\xi_i=0$ for $i<p$, $\xi_p=1$ and, for $i>p$, defined recursively by Formula~\eqref{coef0} 
\begin{equation}\label{coef3}
(i(i+m-2)-p(p+m-2))\xi_i + \sum_{j=0}^{i-1}(j\alpha_{i-j}-\beta_{i-j})\xi_j=0.
\end{equation}
If $i>p$ then $i(i+m-2)-p(p+m-2)>0$, so that the positivity of $\xi_i$ follows by \eqref{coef3} once we  show that $j\alpha_{i-j}-\beta_{i-j}\le 0$ for any $j=0,\dots,i-1$. As $M$ is the round sphere the coefficients $\alpha_j$ and the $\beta_j$ are given by
\[
r(m-1)\cot(r)=\sum_{j\ge0} \alpha_jr^j\ \ \ \mbox{ and }\ \ \  r^2\frac{p(p+m-2)}{\sin^2(r)}=\sum_{j\ge0} \beta_jr^j.
\]
Since $\cot'(r)=-\sin(r)^{-2}$, the above equalities give
\[
\sum_{j\ge0}\frac{j-1}{m-1}\alpha_j r^{j-2}=-\sum_{j\ge0} \frac{\beta_j}{p(p+m-2)}r^{j-2},
\]
so that $\frac{(j-1)p(p+m-2)}{m-1}\alpha_j =-\beta_j$, and
$j\alpha_{i-j}-\beta_{i-j}=\alpha_{i-j}\left(j+\frac{(i-j-1)p(p+m-2)}{m-1}\right)$.
We conclude that $j\alpha_{i-j}-\beta_{i-j}\le 0$ for $j\le i-1$ since  the coefficients of the power series expansion of the cotangent are all non-positive,  but the first one (cf. \cite[chapter 23]{Aigner-Ziegler}).

As a consequence of the non-negativity of the coefficients $\xi_i$ we have that, for any $k\in\N$, the functions $a_k$ and $a'_k$ are non-negative, which easily entails that $C_k^+(R)\le C_k^-(R)$.

%
%
\medskip
\noindent \textbf{Claim 2:} \textit{ for all $\frac{\pi}{2}<R<\pi$ we have $\sigma_1(\Omega_R) = C_1^+(R)$.}

Following Claim 1, we just have to show that $C_1^+(R)\le C_k^+(R)$ for every $k\ge1$.
By Equation \eqref{liouville} we have that for $0<r<\pi$
\begin{equation}\label{eqn-ODE_a_k}
a_k''(r)+(m-1)h(r)a_k'(r)+\lambda_{c,k}h'(r)a_k(r)=0,
\end{equation} 
where $\lambda_{c,k}$ is the $k-$th eigenvalue of $\S^{m-1}$. In particular, for $k=1$ we have $\lambda_{c,k}=m-1$. Integrating equation \eqref{eqn-ODE_a_k} between $r_0$ and $r$ and letting $r_0$ tend to $0$ gives
\begin{equation}\label{eqn-derivative-a_1}
a'_1(r)=m-(m-1)h(r)a_1(r).
\end{equation}
A straightforward computation then brings to
\begin{equation}\label{sigmaO}
C_1^+(R)=\frac{a_1'(R)-a_1'(\pi-R)}{a_1(R)+a_1(\pi-R)}=(m-1)\frac{\cos(\pi-R)}{\sin(\pi-R)}.
\end{equation} 

For $1\le k\le m$ we have $a_k=a_1$ as the function $f_k$ is a first spherical harmonic of $\S^{m-1}$, therefore we have $C_k^+(R)=C_1^+(R)$. Assume now that $k\ge m+1$. In particular, there exists $p\ge 2$ such that $\lambda_{c,k}=p(p+m-2)>m-1$ and $a_k(r)\sim_{0}r^p$. Integrating equation \eqref{eqn-ODE_a_k} between $r_0>0$ and $r>r_0$ we get
\[
a_k'(r)-a_k'(r_0)+(m-1)\int_{r_0}^rh(t)a_k'(t)\,dt+\lambda_{c,k}\int_{r_0}^rh'(t)a_k(t)\,dt=0
\]
which becomes, after an integration by parts,
\[
a_k'(r)-a_k'(r_0)+(m-1)\left(h(r)a_k(r)-h(r_0)a_k(r_0)\right)+\left(\lambda_{c,k}-(m-1)\right) \int_{r_0}^rh'(t)a_k(t)\,dt=0.
\]
Since $a_k(r)=r^p+o(r^p)$, for $r\sim0$,  with $p\ge2$, we can pass to the limit  as $r_0$ tends to $0$ to get
\[
a_k'(r)=-(m-1)h(r)a_k(r)-(\lambda_{c,k}-(m-1))\int_0^rh'(t)a_k(t)\,dt.
\]
Plugging this expression in that of  $C_k^+(R)$ and using that $h(R)=-h(\pi-R)$ we get
\begin{eqnarray}
	C_k^+(R) & = & \frac{a_k'(R)-a_k'(\pi-R)}{a_k(R)+a_k(\pi-R)} \nonumber \\
	& = & \frac{-(m-1)h(R)a_k(R)+(m-1)h(\pi-R)a_k(\pi-R)}{a_k(R)+a_k(\pi-R)} \nonumber \\
	& &	-\frac{\lambda_{c,k}-(m-1)}{a_k(R)+a_k(\pi-R)}\int_{\pi-R}^Rh'(t)a_k(t)\,dt \nonumber \\
	& = & \frac{(m-1)h(\pi-R)(a_k(R)+a_k(\pi-R))}{a_k(R)+a_k(\pi-R)} \nonumber \\
	& &	-\frac{\lambda_{c,k}-(m-1)}{a_k(R)+a_k(\pi-R)}\int_{\pi-R}^Rh'(t)a_k(t)\,dt \nonumber \\
	& = & C_1^+(R)+\frac{\lambda_{c,k}-(m-1)}{a_k(R)+a_k(\pi-R)}\int_{\pi-R}^R\frac{a_k(t)}{\sin^2(t)}\,dt \nonumber \\
	& \ge & C_1^+(R). \nonumber
\end{eqnarray}

This concludes the proof that $\sigma_1(\Omega_R)=C_1^+(R)$.

%
%
\medskip
\noindent \textbf{Claim 3:} \textit{for $R$ close enough to $\pi$ we have $\sigma_1(\Omega_R) > \sigma_1(B_S)$, where $S$ is such that $|B_S|=|\Omega_R|$.}

Let $B_S$ the ball of radius $S$ with the volume $|B_S|=|\Omega_R|$. As $|\S^m\setminus B_S|=|\S^m\setminus \Omega_R|$ we have
\begin{equation}\label{0}
\int_0^{\pi-S} \sin^{m-1}(t)\,dt=2\int_0^{\pi-R} \sin^{m-1}(t)\,dt.
\end{equation}
From Claim 2 we have $\sigma_1(\Omega_R)=(m-1)\frac{\cos(\alpha)}{\sin(\alpha)}.$ where $\alpha=\pi-R$.
Using equation \eqref{eqn-derivative-a_1} we have
\[
\begin{aligned}
\sigma_1(B_S)&=\frac{a'_1(S)}{a_1(S)} \\
&=\frac{m}{a_1(S)}-(m-1)\frac{\cos(S)}{\sin(S)}\\
&=\frac{m\sin^{m-1}(S)}{\int_0^S\sin^{m-1}(t)dt}-(m-1)\frac{\cos(S)}{\sin(S)}\\
&=\frac{m\sin^{m-1}(\beta)}{2\int_0^{\pi-\beta} \sin^{m-1}(t)\,dt}+(m-1)\frac{\cos(\beta)}{\sin(\beta)},
\end{aligned}
\]
where in the last equality we set $\beta=\pi-S$. Notice that $\alpha$ and $\beta$ tend to $0$ as $R$ tends to $\pi$, so that we have $\sigma_1(B_S)\sim (m-1)\frac{\cos(\beta)}{\sin(\beta)}$.
Moreover, as a consequence of \eqref{0}, we get that
\[
\frac{d\beta}{d\alpha}\sin(\beta)^{m-1}=2\sin(\alpha)^{m-1}
\]
which easily entails that $\beta\sim 2^\frac1m \alpha$ (exploiting the fact that every function in role is positive). Those information boil up, after a simple computation, into
\[
\lim_{R\to\pi}\frac{\sigma_1(\Omega_R)}{\sigma_1(B_S)} = 2^\frac1m > 1. 
\]
This concludes the proof of the fact that, for $R<\pi$, $R$ close enough to $\pi$, the ball does not maximize the first Steklov eigenvalue on the sphere.
\end{proof}

%
%
%
%

\small
\bibliographystyle{plain}

\bibliography{steklovmanifolds}

\vspace{10mm}

\begin{flushleft}
	Philippe Castillon \\
	\textsc{imag} (\textsc{u.m.r. c.n.r.s.} 5149) \\
	D\'ept. des Sciences Math\'ematiques, CC 51 \\
	Univ. de Montpellier \\
	34095 \textsc{Montpellier} Cedex 5, France \\
	\texttt{philippe.castillon\symbol{64}umontpellier.fr}
\end{flushleft}

\vspace{5mm}

\begin{flushleft}
	Berardo Ruffini \\
	\textsc{imag} (\textsc{u.m.r. c.n.r.s.} 5149) \\
	D\'ept. des Sciences Math\'ematiques, CC 51 \\
	Univ. de Montpellier \\
	34095 \textsc{Montpellier} Cedex 5, France \\
	\texttt{berardo.ruffini\symbol{64}umontpellier.fr}
\end{flushleft}

\end{document}